\documentclass{amsart}
\usepackage{amssymb}
\numberwithin{equation}{section}

\DeclareMathOperator*{\Argmin}{argmin}

\newcommand{\CAT}{\textup{CAT}}

\theoremstyle{plain}
\newtheorem{theorem}{Theorem}[section]
\newtheorem{lemma}[theorem]{Lemma}
\newtheorem{proposition}[theorem]{Proposition}
\newtheorem{corollary}[theorem]{Corollary}
\newtheorem{example}[theorem]{Example}
\theoremstyle{definition}
\newtheorem{definition}[theorem]{Definition}

\theoremstyle{remark}
\newtheorem{remark}[theorem]{Remark}
\allowdisplaybreaks
\title[Halpern iteration on a geodesic space]{Halpern iteration for a finite family of quasinonexpansive mappings on a complete geodesic space with curvature bounded above by one}
\author[T.~Ezawa]{Tatsuki~Ezawa}
\address[T.~Ezawa]
{Graduate School of Mathematics, Nagoya University, Chikusa-ku, Nagoya 464-8602, Japan}
\email{m14006q@math.nagoya-u.ac.jp}
\author[Y.~Kimura]{Yasunori~Kimura}
\address[Y.~Kimura]
{Department of Information Science, Toho University, Miyama, Funabashi, Chiba 274-8510, Japan}
\email{yasunori@is.sci.toho-u.ac.jp}
\keywords{$\CAT(1)$ space, Halpern iteration, quasinonexpansive, $\Delta$-demicloed, $W$-mapping, fixed point}
\begin{document}
\begin{abstract}
In this paper, we consider the Halpern iteration scheme for a finite family of quasinonexpansive mappings and then prove a strong convergence theorem to their common fixed point in a complete geodesic space with
curvature bounded above by one.
\end{abstract}
\maketitle

\section{Introduction}\label{sec:introduction}

Let us begin with a historical explanation on Halpern schemes. In 1967, Halpern \cite{Halpern} considered an iterative method to find a fixed point of a nonexpansive mapping from the unit ball of a real Hilbert space into itself. In 1992, Wittmann \cite{Wittmann} considered the
following Halpern type iteration scheme in a real Hilbert space $H$: Let $C\subset H$ be a closed convex subset, and $u, x_1\in C$ are given. The iteration scheme is
\[
x_{n+1}:=\alpha_nu+(1-\alpha_n)Tx_n
\]
for all $n\in\mathbb{N}$, where $T$ is a nonexpansive mapping from $C$ into itself such that the set $F(T)$ of its fixed points is nonempty, and where the real sequence $\{\alpha_n\}$ satisfies $\lim_{n\to\infty}\alpha_n=0$, $\sum_{n=1}^\infty\alpha_n=\infty$ and $\sum_{n=1}^\infty|\alpha_{n+1}-\alpha_n|<\infty$. He showed that $\{x_n\}$ converges strongly to a fixed point which is nearest to $u$ in $F(T)$.

In 1997, Shioji-Takahashi \cite{Shioji-Takahashi} extended Wittmann's result to the case where the Hilbert space $H$ is replaced by a Banach space. In 1998, motivated by results of Ishikawa \cite{Ishikawa} and Das-Debata \cite{Das-Debata}, Atsushiba-Takahashi \cite{Atsushiba-Takahashi} considered a variation of Halpern iteration
using $W$-mappings $\{W_n\}$ (see Definition \ref{Wmapping}) in a Banach space: $u, x_1$ are given and
\[
x_{n+1}:=\beta_nu+(1-\beta_n)W_nx_n
\]
for all $n\in\mathbb{N}$.

A $\CAT(0)$ space is a generalization of Hilbert space in a directon different from that of a Banach space. In 2011, Saejung \cite{Saejung} considered the Halpern iteration using single nonexpansive mapping in a $\CAT(0)$ space. In 2011, Phuengrattana-Suantai \cite{Phuen-Suantai} considered the same iteration scheme using $W$-mapping in a convex metric space. Remark that a $\CAT(0)$ space is a convex metric space, so that their result covers the case of CAT(0) space. In 2013, Kimura-Sat\^o \cite{Kimura-Sato} considered the Halpern iteration using single strongly quasinonexpansive mapping in a $\CAT(1)$ space. Remark that a $\CAT(0)$ space is not necessarily a convex metric space.

In this paper, we consider the Halpern iteration with $W$-mapping generated by a finite family of quasinonexpansive mappings in a $\CAT(1)$ space, that is, we showed the following theorem under the similar condition in the result of Kimura-Sat\^{o}:

\begin{theorem}
Let $X$ be a complete $\CAT(1)$ space such that $d(v,v')<\pi/2$ for every $v,v'\in X$. Let $T_1,T_2,\ldots,T_r$ be a finite number of quasinonexpansive and $\Delta$-demiclosed mappings of X into itself such that $F:=\bigcap_{i=1}^rF(T_i)\neq\emptyset$, and let $\alpha_{n,1},a_{n,2},\ldots,\alpha_{n,r}$ be real numbers for $n\in\mathbb{N}$ such that $\alpha_{n,i}\in[a,1-a]$ for every $i=1,2,\ldots,r$, where $0<a<1/2$. Let $W_n$ be the W-mappings of $X$ into itself generated by $T_1,T_2,\ldots,T_r$ and $\alpha_{n,1},\alpha_{n,2},\ldots,\alpha_{n,r}$ for $n\in\mathbb{N}$. Let $\{\beta_n\}$ be a sequence of real numbers such that $0<\beta_n<1$ for every $n\in\mathbb{N}, \lim_{n\to\infty}\beta_n=0$ and $\sum_{n=1}^\infty\beta_n=\infty$. For a given points $u,x_1\in X$, let $\{x_n\}$ be a sequence in $X$ generated by
\[
x_{n+1}=\beta_nu\oplus(1-\beta_n)W_nx_n
\]
for $n\in\mathbb{N}$. Suppose that one of the following conditions holds$:$
\begin{enumerate}
\item[(a)] $\sup_{v,v'\in X}d(v,v')<\pi/2;$
\item[(b)] $d(u,P_Fu)<\pi/4$ and $d(u,P_Fu)+d(x_1,P_Fu)<\pi/2;$
\item[(c)] $\sum_{n=1}^\infty\beta_n^2=\infty.$
\end{enumerate}
Then $\{x_n\}$ converges to $P_Fu$.
\end{theorem}

The proof will be given in \S \ref{sec:Main result}.

In \S \ref{sec:Applications}, we give some applications of the main theorem. In Theorem \ref{Convex function}, we give an approximation of a minimizer of convex functions on a complete $\CAT(1)$ space. A further application will be given in Theorem \ref{Metric projection}. We also give an example of quasinonexpansive mappings which is not strongly quasinonexpansive in Example \ref{example}.

\section{Preliminaries}\label{sec:Preliminaries}

Let $(X,d)$ be a metric space. For $x,y\in X$, a mapping $c:[0,l]\to X$ is a geodesic of $x,y\in X$ if $c(0)=x,c(l)=y$ and $d(c(s), c(t))=|s-t|$ for all $s,t\in[0,l]$. For $r>0$, if a geodesic exists for every $x, y\in X$ with $d(x,y)<r$, then $X$ is called an $r$-geodesic metric space. If a geodesic is unique for every $x,y\in X$, we define $[x,y]:=c([0,l])$ and it is called a geodesic segment of $x,y\in X$. In what follows, a metric space $X$ is always assumed to be $\pi$-geodesic and every geodesic is unique. For $x,y\in X$, let $c:[0,l]\to X$ be a geodesic of $x,y\in X$. For $t\in[0,l]$, we denote
\[
tx\oplus(1-t)y:=c((1-t)l).
\]
In other words, $z:=tx\oplus(1-t)y$ satisfies $d(x,z)=(1-t)d(x,y)$. Let $X$ be a geodesic metric space. A geodesic triangle is defined by the union of segment $\triangle(x,y,z):=
[x,y]\cup [y,z]\cup [z,x]$. Let $\mathbb{S}^2$ be the unit sphere of the Euclidean space $\mathbb{R}^3$ and $d_{\mathbb{S}^2}$ is the spherical metric on $\mathbb{S}^2$. Then, for $x,y,z\in X$ satisfying $d(x,y)+d(y,z)+d(z,x)<2\pi$, there exist $\overline{x},\overline{y},\overline{z}\in \mathbb{S}^2$ such that $d(x,y)=d_{\mathbb{S}^2}(\overline{x},\overline{y}), d(y,z)=d_{\mathbb{S}^2}(\overline{y},\overline{z})$ and $d(z,x)=d_{\mathbb{S}^2}(\overline{z},\overline{x})$. A point $\overline{p}\in[\overline{x},\overline{y}]$ is called a comparison point for $p\in[x,y]$ if $d_{\mathbb{S}^2}(\overline{x},\overline{p})=d(x,p)$. If every $p,q$ on the triangle $\triangle(x,y,z)$ with $d(x,y)+d(y,z)+d(z,x)<2\pi$ and their comparison points $\overline{p},\overline{q}\in \triangle(\overline{x},\overline{y},\overline{z})$ satisfy that
\[
d(p,q)\leq d_{\mathbb{S}^2}(\overline{p},\overline{q}),
\]
$X$ is called a $\CAT(1)$ space. We refer details and examples of a $\CAT(1)$ space to \cite{Bridson-Haefliger}.

\begin{theorem}[Kimura-Sat\^o \cite{KS2}]\label{pal}
Let $x,y,z$ be points in {\rm CAT(1)} space such that $d(x,y)+d(y,z)+d(z,x)<2\pi$. Let $v:=tx\oplus(1-t)y$ for some $t\in[0,1]$. Then
\[
\cos d(v,z)\sin d(x,y)\geq \cos d(x,z)\sin(td(x,y)) + \cos d(y, z)\sin((1-t)d(x, y)).
\]
\end{theorem}

\begin{corollary}[Kimura-Sat\^o \cite{Kimura-Sato}]\label{pall}
Let $x,y,z$ be points in $\CAT(1)$ space such that $d(x,y)+d(y,z)+d(z,x)<2\pi$. Let $v:=tx\oplus(1-t)y$ for some $t\in[0,1]$. Then
\[
\cos d(v,z)\geq t\cos d(x,z) + (1-t)\cos d(y,z).
\]
\end{corollary}

Let $X$ be a complete $\CAT(1)$ space such that $d(v,v')<\pi/2$ for all $v,v'\in X$, and let $C$ be a nonempty closed convex subset of $X$. Then for any $x\in X$, there exists a unique point $P_Cx\in C$ such that
\[
d(x,P_Cx)=\inf_{y\in C}d(x,y).
\]
That is, using similar techniques to the cace of Hilbert space, we can define metric projection $P_C$ from $X$ onto $C$ such that $P_Cx$ is the nearest point of $C$ to $x$. Let $X$ be a metric space and $\{x_n\}$ a bounded sequence of $X$. The asymptotic center $AC(\{x_n\})$ of $\{x_n\}$ is defined by
\[
AC(\{x_n\}):=\left\{ z\ |\ \limsup_{n\to\infty}d(z,x_n)=\inf_{x\in X}\limsup_{n\to\infty}d(x,x_n)\right\}.
\]
We say that $\{x_n\}$ is $\Delta$-convergent to a point $z$ if for all subsequences $\{x_{n_i}\}$ of $\{x_n\}$, its asympotic center consists only of $z$, that is,  $AC(\{x_{n_i}\})=\{z\}$. Let $X$ be a metric space. Let $T$ be a mapping of $X$ into itself. Then, $T$ is said to be nonexpansive if $d(Tx,Ty)\leq d(x,y)$ for all $x,y\in X$. Hereafter we denote
\[
F(T):=\{z\ |\ Tz=z\}
\]
the set of fixed points. Then $T$ is said to be quasinonexpansive if $d(Tx, p)\leq d(x, p)$ for all $x\in X$ and $p\in F(T)$. Using similar techniques to the case of Hilbert space, we can prove that $F(T)$ is a closed convex subset of $X$. $T$ is said to be strongly quasinonexpansive if it is quasinonexpansive, and for every $p\in F(T)$ and every sequence in $X$ satisfying that $\sup_{n\in\mathbb{N}}d(x_n,p)<\pi/2$ and $\lim_{n\to\infty}(\cos d(x_n, p)/\cos d(Tx_n, p))=0$, it follows that $\lim_{n\to\infty}d(x_n,Tx_n) =0$. $T$ is said to be $\Delta$-demiclosed if for any $\Delta$-convergent sequence $\{x_n\}$ in $X$, its $\Delta$-limit belongs to $F(T)$ whenever $\lim_{n\to\infty}d(Tx_n,x_n)=0$.

The notation of $W$-mapping is originally proposed by Takahashi. We use the same notation in the setting of geodesic space as following:

\begin{definition}[Takahashi \cite{Takahashi}]\label{Wmapping}\normalfont
Let $X$ be a geodesic metric space. Let $T_1,T_2,\ldots,T_r$ be a finite number of mappings of $X$ into itself and $\alpha_1,\alpha_2,\ldots,\alpha_r$ be real numbers such that $0\leq \alpha_i\leq 1$ for every $i=1,2,\ldots,r$. Then, we define a mapping $W$ of $X$ into itself as follows$:$
\begin{align*}
U_1 &:=  \alpha_1T_1\oplus(1-\alpha_1)I,  \\
U_2 &:= \alpha_2T_2U_1\oplus(1-\alpha_2)I,  \\
 &\cdots& \\
U_r &:= \alpha_rT_rU_{r-1}\oplus(1-\alpha_r)I, \\
W &:= U_r.
\end{align*}
Such a mapping $W$ is called a $W$-mapping generated by $T_1,T_2,\ldots,T_r$ and $\alpha_1,\alpha_2,\ldots,\alpha_r$.
\end{definition}

The following lemmas are important for our main result.

\begin{lemma}[Kimura-Sat\^o \cite{Kimura-Sato}]\label{quasinonexpansive}
Let $T$ be a quasinonexpansive mapping defined on a {\rm CAT(1)} space. For any real number $\alpha\in[0,1]$, the mapping $\alpha T\oplus(1-\alpha)I$ is quasinonexpansive.
\end{lemma}

The proof of Lemma \ref{quasinonexpansive} is essentially obtained in \cite{Kimura-Sato}, so we omit the proof.

\begin{lemma}[Kimura-Sat\^o \cite{Kimura-Sato}]\label{delta-demiclosed}
Let $T$ be a nonexpansive mapping on a {\rm CAT(1)} space. For a any real number $\alpha\in(0,1]$, the mapping $\alpha T\oplus(1-\alpha)I$ is $\Delta$-demiclosed.
\end{lemma}

\begin{lemma}[Saejung-Yotkaew \cite{Saejung-Yatkew}]\label{bankoku}
Let $\{s_n\}, \{t_n\}$ be sequences of real numbers  such that $s_n\geq 0$  for every $n\in\mathbb{N}$. Let $\{\gamma_n\}$ be a sequence in $(0,1)$ such that $\sum_{n=0}^\infty\gamma_n=\infty$. Suppose that $s_{n+1}\leq(1-\gamma_n)s_n+\gamma_nt_n$ for every $n\in\mathbb{N}$. If $\limsup_{j\to\infty}t_{n_j}\leq 0$ for every subsequence $\{n_j\}$ of $\mathbb{N}$ satisfying $\liminf_{j\to\infty}(s_{n_j+1}-s_{n_j})\geq 0$, then $\lim_{n\to\infty}s_n=0$.
\end{lemma}

\begin{lemma}[Esp\'\i nola-Fern\'andez-Le\'on \cite{Espinola-Fernandez-Leon}]\label{subseq}
Let $X$ be a complete $\CAT(1)$ space, and $\{x_n\}$ be a sequence in $X$. If there exists $x\in X$ such that $\limsup_{n\to\infty}d(x_n,x)<\pi/2$, then $\{x_n\}$ has a $\Delta$-convergent subsequence.
\end{lemma}

\begin{lemma}[He-Fang-Lopez-Li \cite{He-Fang-Lopez-Li}]\label{conv}
Let $X$ be a complete $\CAT(1)$ space and $p\in X$. If a
sequence $\{x_n\}$ in $X$ satisfies that $\limsup_{n\to\infty} d(x_n,p)<\pi/2$ and that $\{x_n\}$ is $\Delta$-convergent to
$x\in X$, then $d(x, p)\leq \liminf_{n\to\infty}d(x_n,p)$.
\end{lemma}

\begin{lemma}[Kimura-Sat\^o \cite{Kimura-Sato}]\label{KS}
Let $X$ be a $\CAT(1)$ space such that $d(v,v')<\pi/2$ for every $v,v'\in X$. Let $\alpha\in[0,1]$ and $u,y,z\in X$. Then
\begin{align*}
& 1-\cos d(\beta u\oplus(1-\beta)y,z)  \\
&\leq (1-\gamma)(1-\cos d(y,z))+\gamma\left(1-\dfrac{\cos d(u,z)}{\sin d(u,y)\tan(2^{-1}\beta d(u,y))+\cos d(u,y)}\right),
\end{align*}
where
\[
\gamma:=\left\{ \begin{array}{ll}
1-\dfrac{\sin((1-\beta)d(u,y))}{\sin(\beta d(u,y))} & (u\neq y), \\
\beta & (u=y).
\end{array} \right.
\]
\end{lemma}

\section{Main result}\label{sec:Main result}

We begin this section with the following useful lemma.

\begin{lemma}\label{sin}
If $\delta\in[0,\pi/2]$ satisfies
\[
\sin\delta\geq \sin(\alpha\delta)+\sin((1-\alpha)\delta)
\]
for some $\alpha\in(0,1)$, then $\delta=0$.
\end{lemma}

\begin{proof}
It is obtained by an elementary calculation.
\end{proof}

Next we study the set of fixed points of a $W$-mapping.

\begin{proposition}\label{FixedPoint}
Let $X$ be a $\CAT(1)$ space. Let $T_1,T_2,\ldots,T_r$ be quasinonexpansive mappings of X into itself such that $\bigcap_{i=1}^rF(T_i)\neq\emptyset$ and let $\alpha_1,\alpha_2,\ldots,\alpha_r$ be real numbers such that $0<\alpha_i<1$ for every $i=1,2,\ldots,r$. Let $W$ be the $W$-mappig of $X$ into itself generated by $T_1,T_2,\ldots,T_r$ and $\alpha_1,\alpha_2,\ldots,\alpha_r$. Then, $F(W)=\bigcap_{i=1}^rF(T_i)$.
\end{proposition}

\begin{proof}
It is obvious that $\bigcap_{i=1}^rF(T_i)\subset F(W)$. So, we shall prove $F(W)\subset\bigcap_{i=1}^rF(T_i)$. Let $z\in F(W)$ and $w\in \bigcap_{i=1}^rF(T_i)$. Then it follows that
\[
0=d(z,z)=d(Wz,z)=d(\alpha_rT_rU_{r-1}z\oplus(1-\alpha_r)z,z)=\alpha_rd(z,T_rU_{r-1}z).
\]
Since $0<\alpha_r\leq 1$, we obtain $z=T_rU_{r-1}z$ and hence
\begin{align*}
\cos d(z,w) &= \cos d(T_rU_{r-1}z,w) \\
&\geq \cos d(U_{r-1}z,w) \\
&= \cos d(\alpha_{r-1}T_{r-1}U_{r-2}z\oplus(1-\alpha_{r-1})z,w) \\
&\geq \alpha_{r-1}\cos d(T_{r-1}U_{r-2}z,w)+(1-\alpha_{r-1})\cos d(z,w) \\
&\geq \alpha_{r-1}\cos d(U_{r-2}z,w)+(1-\alpha_{r-1})\cos d(z,w) \\
&\geq \alpha_{r-1}\cos d(\alpha_{r-2}T_{r-2}U_{r-3}z\oplus(1-\alpha_{r-2})z,w)+(1-\alpha_{r-1})\cos d(z,w) \\
& \geq \alpha_{r-1}\alpha_{r-2}\cos d(T_{r-2}U_{r-3}z,w)+(1-\alpha_{r-1}\alpha_{r-2})\cos d(z,w) \\
&\geq \cdots \\
&\geq \alpha_{r-1}\alpha_{r-2}\cdots\alpha_2\cos d(T_2U_1z,w)+(1-\alpha_{r-1}\alpha_{r-2}\cdots\alpha_2)\cos d(z,w) \\
&\geq \alpha_{r-1}\alpha_{r-2}\cdots\alpha_2\cos d(U_1z,w)+(1-\alpha_{r-1}\alpha_{r-2}\cdots\alpha_2)\cos d(z,w) \\
&\geq \alpha_{r-1}\alpha_{r-2}\cdots\alpha_2\cos d(\alpha_1T_1z\oplus(1-\alpha_1)z,w)+(1-\alpha_{r-1}\alpha_{r-2}\cdots\alpha_2)\cos d(z,w) \\
&\geq \alpha_{r-1}\alpha_{r-2}\cdots\alpha_2\alpha_1\cos d(T_1z,w)+(1-\alpha_{r-1}\alpha_{r-2}\cdots\alpha_2\alpha_1)\cos d(z,w) \\
&\geq \cos d(z,w).
\end{align*}
Then it follows that
\[
d(z,w)=\cos d(T_1z,w)=d(U_1z,w)=d(\alpha_1T_1z\oplus(1-\alpha_1)z,w).
\]
By Theorem \ref{pal} and Lemma \ref{sin} with
\begin{align*}
& \cos d(\alpha_1T_1z\oplus(1-\alpha_1)z,w)\sin d(T_1z,z) \\
&\geq \cos d(T_1z,w)\sin(\alpha_1d(T_1z,z))+\cos d(z,w)\sin((1-\alpha_1)d(T_1z,z)),
\end{align*}
we obtain $T_1z=z$. Similarly, we have
\[
d(z,w)=d(T_2U_1z,w)=d(U_2z,w)=d(\alpha_2T_2U_1z\oplus(1-\alpha_2)z,w).
\]
By Theorem \ref{pal} and Lemma \ref{sin} with
\begin{align*}
& \cos d(\alpha_2T_2U_1z\oplus(1-\alpha_2)z,w)\sin d(T_2U_1z,z) \\
&\geq \cos d(T_2U_1z,w)\sin(\alpha_2d(T_2U_1z,z))+\cos d(z,w)\sin((1-\alpha_2)d(T_2U_1z,z)),
\end{align*}
we obtain $T_2U_1z=z$. Since $U_1z=z$, we obtain $T_2z=z$. Using such techniques, we obtain $T_iz=z$ and $U_iz=z$ for all $i=1,2,\ldots,r$, and hence $z\in\bigcap_{i=1}^rF(T_i)$. This implies $F(W)\subset \bigcap_{i=1}^rF(T_i)$. Therefore we have $F(W)=\bigcap_{i=1}^rF(T_i)$.
\end{proof}

\begin{remark}\normalfont
Let $W_n$ be the $W$-mappings of $X$ into itself generated by $T_1,T_2,\ldots,T_r$ and $\alpha_{n,1},\alpha_{n,2},\ldots,\alpha_{n,r}$ for $n\in\mathbb{N}$. By Proposition \ref{FixedPoint}, all the sets of fixed points $\{F(W_n)\}$ is identical.
\end{remark}

The following Lemma \ref{Strongly quasinonexpansive sequence} is essentially given by Kasahara \cite{Kasahara}. For the sake of completeness, we give the proof.

\begin{lemma}[Kasahara \cite{Kasahara}]\label{Strongly quasinonexpansive sequence}
Let $\{S_n\}$ be a sequence of quasinonexpansive mappings of a $\CAT(1)$ space $X$ into itself such that $\bigcap_{n=1}^\infty F(S_n)\neq\emptyset$. Then for given real numbers $\alpha_n\in[a,1-a]\subset(0,1)$ and $p\in\bigcap_
{n=1}^\infty F(S_n)$, if $\{x_n\}$ satisfies that $\sup_{n\in\mathbb{N}}d(x_n,p)<\pi/2$ and
\[
\lim_{n\to\infty}\dfrac{\cos d(x_n,p)}{\cos d(\alpha_nS_nx_n\oplus(1-\alpha_n)x_n,p)}=1,
\]
then $\lim_{n\to\infty}d(S_nx_n,x_n)=0$.
\end{lemma}

\begin{proof}
Let $\delta_n:=d(S_nx_n,x_n)$. Assume that $\{x_n\}\subset X$ and $p\in \bigcap_{n=1}^\infty F(S_n)$ such that $\sup_{n\in\mathbb{N}}d(x_n,p)<\pi/2$ and $\lim_{n\to\infty}(\cos d(x_n,p)/\cos d(\alpha_nS_nx_n\oplus(1-\alpha_n)x_n,p))=1$, by Theorem \ref{pal}, we have
\begin{align*}
& \cos d(\alpha_nS_nx_n\oplus(1-\alpha_n)x_n,p)\sin d(S_nx_n,x_n) \\
&\geq \cos d(S_nx_n,p)\sin(\alpha d(S_nx_n,x_n))+\cos d(x_n,p)\sin((1-\alpha_n)d(S_nx_n,x_n)) \\
&\geq \min\{ \cos d(S_nx_n,p),\cos d(x_n,p)\}(\sin(\alpha_n d(S_nx_n,x_n))+\sin((1-\alpha_n)d(S_nx_n,x_n))) \\
&=2\cos d(x_n,p)\sin\dfrac{d(S_nx_n,x_n)}{2}\cos \dfrac{(2\alpha_n-1)d(S_nx_n,x_n)}{2}.
\end{align*}
Hence
\[
\cos d(\alpha_nS_nx_n\oplus(1-\alpha_n)x_n,p)\sin\delta_n\geq 2\cos d(x_n,p)\sin\dfrac{\delta_n}{2}\cos \dfrac{(2\alpha_n-1)\delta_n}{2}.
\]
We assume that $\delta_n\neq 0$. Dividing above by $2\sin(\delta_n/2)$, we have
\begin{align*}
\cos d(\alpha_nS_nx_n\oplus(1-\alpha_n)x_n,p)\cos\dfrac{\delta_n}{2} &\geq \cos d(x_n,p)\cos\dfrac{(2\alpha_n-1)\delta_n}{2} \\
&\geq \cos d(x_n,p)\cos\dfrac{(1-2a)\delta_n}{2}.
\end{align*}
Moreover, dividing above by $\cos((1-2a)\delta_n/2)$, we have
\[
\cos d(x_n,p) \leq \cos d(\alpha S_nx_n\oplus(1-\alpha_n)x_n,p)\dfrac{\cos\dfrac{\delta_n}{2}}{\cos\dfrac{(1-2a)\delta_n}{2}}.
\]
Then
\begin{align*}
& \cos d(x_n,p) \\
&\leq \cos d(\alpha_n S_nx_n\oplus(1-\alpha_n)x_n,p)\dfrac{\cos\dfrac{(1-2a)\delta_n}{2}\cos(a\delta_n)-\sin\dfrac{(1-2a)\delta_n}{2}\sin(a\delta_n)}{\cos\dfrac{(1-2a)\delta_n}{2}} \\
&\leq \cos d(\alpha_n S_nx_n\oplus(1-\alpha_n)x_n,p)\cos(a \delta_n).
\end{align*}
Thus we have that
\[
\cos d(a \delta_n)\geq \dfrac{\cos d(x_n,p)}{\cos d(\alpha_n S_nx_n\oplus(1-\alpha_n)x_n,p)}\to 1\ (n\to\infty),
\]
which implies $\lim_{n\to\infty}\delta_n=0$, that is, $\lim_{n\to\infty}d(S_nx_n,x_n)=0$.
\end{proof}

\begin{theorem}\label{Main}
Let $X$ be a complete ${\rm CAT(1)}$ space such that $d(v,v')<\pi/2$ for every $v,v'\in X$. Let $T_1,T_2,\ldots,T_r$ be a finite number of quasinonexpansive and $\Delta$-demiclosed mappings of X into itself such that $F:=\bigcap_{i=1}^rF(T_i)\neq\emptyset$, and let $\alpha_{n,1},a_{n,2},\ldots,\alpha_{n,r}$ be real numbers for $n\in\mathbb{N}$ such that $\alpha_{n,i}\in[a,1-a]$ for every $i=1,2,\ldots,r$, where $0<a<1/2$. Let $W_n$ be the W-mappings of $X$ into itself generated by $T_1,T_2,\ldots,T_r$ and $\alpha_{n,1},\alpha_{n,2},\ldots,\alpha_{n,r}$ for $n\in\mathbb{N}$. Let $\{\beta_n\}$ be a sequence of real numbers such that $0<\beta_n<1$ for every $n\in\mathbb{N}, \lim_{n\to\infty}\beta_n=0$ and $\sum_{n=1}^\infty\beta_n=\infty$. For a given points $u,x_1\in X$, let $\{x_n\}$ be a sequence in $X$ generated by
\[
x_{n+1}=\beta_nu\oplus(1-\beta_n)W_nx_n
\]
for $n\in\mathbb{N}$. Suppose that one of the following conditions holds$:$
\begin{enumerate}
\item[(a)] $\sup_{v,v'\in X}d(v,v')<\pi/2;$
\item[(b)] $d(u,P_Fu)<\pi/4$ and $d(u,P_Fu)+d(x_1,P_Fu)<\pi/2;$
\item[(c)] $\sum_{n=1}^\infty\beta_n^2=\infty.$
\end{enumerate}
Then $\{x_n\}$ converges to $P_Fu$.
\end{theorem}

\begin{proof}
Let $p:=P_Fu$ and let
\begin{align*}
s_n &:= 1-\cos d(x_n,p), \\
t_n &:= 1-\dfrac{\cos d(u,p)}{\sin d(u,W_nx_n)\tan(2^{-1}\beta_nd(u,W_nx_n))+\cos d(u,W_nx_n)}, \\
\gamma_n &:= \left\{ \begin{array}{ll}
1-\dfrac{\sin((1-\beta_n)d(u,W_nx_n))}{\sin(\beta_nd(u,W_nx_n))} & (u\neq W_nx_n), \\
\beta_n & (u=W_nx_n)
\end{array} \right.
\end{align*}
for $n\in\mathbb{N}$. If $\{s_n\},\{t_n\}$ and $\{\gamma_n\}$ satisfy the conditions of Lemma \ref{bankoku}, then we will have $\lim_{n\to\infty}s_n=0$, that is, $\{x_n\}$ converges to $p=P_Fu$. Thus the proof of Theorem \ref{Main} will be completed. First, it is obvious that $s_n\geq 0$. By Lemma \ref{quasinonexpansive}, $W_n$ is quasinonexpansive. Then, it follows from Lemma \ref{KS} that
\[
s_{n+1}\leq(1-\gamma_n)(1-\cos d(W_nx_n,p))+\gamma_nt_n\leq(1-\gamma_n)s_n+\gamma_nt_n
\]
for every $n\in\mathbb{N}$. Now, it is also obvious that $\{\gamma_n\}$ is a sequence in $(0,1)$. we show that $\sum_{n=1}^\infty\gamma_n=\infty$ holds under each condition (a),(b) and (c). We have
\begin{align*}
\cos d(x_{n+1},p) &= \cos d(\beta_nu\oplus(1-\beta_n)W_nx_n,p) \\
&\geq \beta_n\cos d(u,p)+(1-\beta_n)\cos d(W_nx_n,p) \\
&\geq \beta_n\cos d(u,p)+(1-\beta_n)\cos d(x_n,p) \\
&\geq \min\{\cos d(u,p),\cos d(x_n,p)\}
\end{align*}
for all $n\in\mathbb{N}$. Thus we have
\begin{align*}
\cos d(x_n,p) &\geq \min\{\cos d(u,p),\cos d(x_1,p)\} \\
&= \cos\max\{d(u,p),d(x_1,p)\} \\
&>0
\end{align*}
for all $n\in\mathbb{N}$ and hence $\sup_{n\in\mathbb{N}}d(x_n,p)\leq\max\{d(u,p),d(x_1,p)\}<\pi/2$. For the case of (a) and (b), let $M=\sup_{n\in\mathbb{N}}d(u,W_nx_n)$. Then we show that $M<\pi/2$. For (a), it is trivial. For (b), since $\sup_{n\in\mathbb{N}}d(x_n,p)\leq\max\{d(u,p),d(x_1,p)\}$, we have
\begin{align*}
M &= \sup_{n\in\mathbb{N}}d(u,W_nx_n) \\
&\leq \sup_{n\in\mathbb{N}}(d(u,p)+d(p,W_nx_n)) \\
&\leq \sup_{n\in\mathbb{N}}(d(u,p)+d(p,x_n)) \\
&\leq \max\{2d(u,p),d(u,p)+d(x_1,p)\} \\
&<\dfrac{\pi}{2}.
\end{align*}
Thus, in each case of (a) and (b), we have
\begin{align*}
\gamma_n &\geq 1-\dfrac{\sin((1-\beta_n)M)}{\sin M} \\
&=\dfrac{2}{\sin M}\sin\left(\dfrac{\beta_n}{2}M\right)\cos\left(\left(1-\dfrac{\beta_n}{2}\right)M\right) \\
&\geq \beta_n\cos M.
\end{align*}
Since $\sum_{n=1}^\infty\beta_n=\infty$, it follows that $\sum_{n=1}^\infty\gamma_n=\infty$. For the case of (c), we have
\[
\gamma_n\geq 1-\sin\dfrac{(1-\beta_n)\pi}{2}=1-\cos\dfrac{\beta_n\pi}{2}\geq\dfrac{\beta_n^2\pi^2}{16}
\]
for every $n\in\mathbb{N}$. Therefore, in the case of (c) we also have $\sum_{n=1}^\infty\gamma_n=\infty$. Finally, we show that $\limsup_{j\to\infty}t_{n_j}\leq 0$ for any subsequence $\{n_j\}$ of $\mathbb{N}$ with $\liminf_{j\to\infty}(s_{n_j+1}-s_{n_j})\geq 0$.  Let $\{s_{n_j}\}$ be a subsequence of $\{s_n\}$ satisfying that $\liminf_{j\to\infty}(s_{n_j+1}-s_{n_j})\geq 0$, and put
\[
\alpha:=\min_{k=1,\ldots,r}\left(\inf_{n\in\mathbb{N}}\alpha_{n,k}\right).
\]
Then we have
\begin{align*}
0 &\leq \liminf_{j\to\infty}(s_{n_j+1}-s_{n_j}) \\
&= \liminf_{j\to\infty}(\cos d(x_{n_j},p)-\cos d(x_{n_j+1},p)) \\
&= \liminf_{j\to\infty}(\cos d(x_{n_j},p)-\cos d(\beta_{n_j}u\oplus(1-\beta_{n_j})W_{n_j}x_{n_j},p)) \\
&\leq \liminf_{j\to\infty}(\cos d(x_{n_j},p)-(\beta_{n_j}\cos d(u,p)+(1-\beta_{n_j})\cos d(W_{n_j}x_{n_j},p)) \\
&= \liminf_{j\to\infty}(\cos d(x_{n_j},p)-\cos d(W_{n_j}x_{n_j},p)) \\
&= \liminf_{j\to\infty}(\cos d(x_{n_j},p)-\cos d(\alpha_{n_j,r}T_rU_{n_j,r-1}x_{n_j}\oplus(1-\alpha_{n_j,r})x_{n_j},p)) \\
&\leq \liminf_{j\to\infty}(\cos d(x_{n_j},p)-(\alpha_{n_j,r}\cos d(T_rU_{n_j,r-1}x_{n_j},p)+(1-\alpha_{n_j,r})\cos d(x_{n_j},p))) \\
&= \liminf_{j\to\infty}(\alpha_{n_j,r}\cos d(x_{n_j},p)-\alpha_{n_j,r}\cos d(T_rU_{n_j,r-1}x_{n_j},p)) \\
&\leq \alpha\liminf_{j\to\infty}(\cos d(x_{n_j},p)-\cos d(T_rU_{n_j,r-1}x_{n_j},p)) \\
&\leq \alpha\liminf_{j\to\infty}(\cos d(x_{n_j},p)-\cos d(U_{n_j,r-1}x_{n_j},p)) \\
&= \alpha\liminf_{j\to\infty}(\cos d(x_{n_j},p)-\cos d(\alpha_{n_j,r-1}T_{r-1}U_{n_j,r-2}x_{n_j}\oplus(1-\alpha_{n_j,r-1})x_{n_j},p)) \\
&\leq \alpha\liminf_{j\to\infty}(\cos d(x_{n_j},p)-(\alpha_{n_j,r-1}\cos d(T_{r-1}U_{n_j,r-2}x_{n_j},p)+(1-\alpha_{n_j,r-1})\cos d(x_{n_j},p))) \\
&= \alpha\liminf_{j\to\infty}(\alpha_{n_j,r-1}\cos d(x_{n_j},p)-\alpha_{n_j,r-1}\cos d(T_{r-1}U_{n_j,r-2}x_{n_j},p)) \\
&\leq \alpha^2\liminf_{j\to\infty}(\cos d(x_{n_j},p)-\cos d(T_{r-1}U_{n_j,r-2}x_{n_j},p)) \\
&\leq \cdots \\
&\leq \alpha^{r-1}\liminf_{j\to\infty}(\cos d(x_{n_j},p)-\cos d(T_2U_{n_j,1}x_{n_j},p)) \\
&\leq \alpha^{r-1}\liminf_{j\to\infty}(\cos d(x_{n_j},p)-\cos d(U_{n_j,1}x_{n_j},p)) \\
\end{align*}
\begin{align*}
&= \alpha^{r-1}\liminf_{j\to\infty}(\cos d(x_{n_j},p)-\cos d(\alpha_{n_j,1}T_1x_{n_j}\oplus(1-\alpha_{n_j,1})x_{n_j},p)) \\
&\leq\alpha^{r-1}\limsup_{j\to\infty}(\cos d(x_{n_j},p)-\cos d(\alpha_{n_j,1}T_1x_{n_j}\oplus(1-\alpha_{n_j,1})x_{n_j},p)) \\
&\leq 0.
\end{align*}

Thus we have
\[
\lim_{j\to\infty}(\cos d(x_{n_j},p)-\cos d(\alpha_{n_j,1}T_1x_{n_j}\oplus(1-\alpha_{n_j,1})x_{n_j},p))=0.
\]
Using the inequality $\sup_{j\in\mathbb{N}}d(x_{n_j},p)<\pi/2$, we also have
\[
\lim_{j\to\infty}\dfrac{\cos d(x_{n_j},p)}{\cos d(\alpha_{n_j,1}T_1x_{n_j}\oplus(1-\alpha_{n_j,1})x_{n_j},p)}=1.
\]
By Lemma \ref{Strongly quasinonexpansive sequence}, it follows that
\[
\lim_{j\to\infty}d(T_1x_{n_j},x_{n_j})=0.
\]
Put
\[
y_j^{(k)}:=U_{n_j,k}x_{n_j}
\]
for $k=1,2,\ldots,r-1$. We show that
\[
\lim_{j\to\infty}d(x_{n_j},y_j^{(k)})=0,\quad \lim_{j\to\infty}d(T_{k+1}y_j^{(k)},y_j^{(k)})=0
\]
by induction on $k=1,2,\ldots,r-1$. First, we consider the case $k=1$. We have
\begin{align*}
\lim_{j\to\infty}d(x_{n_j},y_j^{(1)}) &= \lim_{j\to\infty}d(x_{n_j},U_{n_j,1}x_{n_j}) \\
&= \lim_{j\to\infty}d(x_{n_j}, \alpha_{n_j,1}T_1x_{n_j}\oplus(1-\alpha_{n_j,1})x_{n_j}) \\
&=\lim_{j\to\infty}\alpha_{n_j}d(T_1x_{n_j},x_{n_j}) \\
&=0.
\end{align*}
On the other hand, by the calculation above we have
\begin{align*}
0 &\leq \liminf_{j\to\infty}(\cos d(x_{n_j},p)-\cos d(U_{n_j,2}x_{n_j},p)) \\
&= \liminf_{j\to\infty}(\cos d(x_{n_j},p)-\cos d(\alpha_{n_j,2}T_2U_{n_j,1}x_{n_j}\oplus(1-\alpha_{n_j,2})x_{n_j},p)) \\
&\leq \limsup_{j\to\infty}(\cos d(x_{n_j},p)-\cos d(\alpha_{n_j,2}T_2U_{n_j,1}x_{n_j}\oplus(1-\alpha_{n_j,2})x_{n_j},p)) \\
&\leq 0.
\end{align*}
Therefore
\[
\lim_{j\to\infty}(\cos d(x_{n_j},p)-\cos d(\alpha_{n_j,2}T_2U_{n_j,1}x_{n_j}\oplus(1-\alpha_{n_j,2})x_{n_j},p))=0.
\]
Using the inequality $\sup_{j\in\mathbb{N}}d(x_{n_j},p)<\pi/2$, we also have
\[
\lim_{j\to\infty}\dfrac{\cos d(x_{n_j},p)}{\cos d(\alpha_{n_j,2}T_2U_{n_j,1}x_{n_j}\oplus(1-\alpha_{n_j,2})x_{n_j},p)}=1.
\]
By Lemma \ref{Strongly quasinonexpansive sequence}, and since $\lim_{j\to\infty}d(x_{n_j},y_j^{(1)})=0$,
\[
\lim_{j\to\infty}d(T_2y_j^{(1)},y_j^{(1)})\leq \lim_{j\to\infty}(d(T_2y_j^{(1)},x_{n_j})+d(x_{n_j},y_j^{(1)}))=0.
\]
Hence we have that case $k=1$, that is,
\[
\lim_{j\to\infty}d(x_{n_j},y_j^{(1)})=0,\quad \lim_{j\to\infty}d(T_2y_j^{(1)},y_j^{(1)})=0.
\]
holds. Next, assume the hypothesis with $k=l$, that is,
\[
\lim_{j\to\infty}d(x_{n_j},y_j^{(l)})=0,\quad \lim_{j\to\infty}d(T_{l+1}y_j^{(l)},y_j^{(l)})=0
\]
holds. Then by assumption, we have
\begin{align*}
\lim_{j\to\infty}d(x_{n_j},y_j^{(l+1)}) &= \lim_{j\to\infty}d(x_{n_j},U_{n_j,l+1}x_{n_j}) \\
&= \lim_{j\to\infty}d(x_{n_j}, \alpha_{n_j,l+1}T_{l+1}U_{n_j,l}x_{n_j}\oplus(1-\alpha_{n_j,l+1})x_{n_j}) \\
&= \lim_{j\to\infty}d(x_{n_j}, \alpha_{n_j,l+1}T_{l+1}y_j^{(l)}\oplus(1-\alpha_{n_j,l+1})x_{n_j}) \\
&= \lim_{j\to\infty}d(x_{n_j}, \alpha_{n_j,l+1}y_j^{(l)}\oplus(1-\alpha_{n_j,l+1})x_{n_j}) \\
&= \lim_{j\to\infty}\alpha_{n_j,l+1}d(x_{n_j},y_j^{(l)}) \\
&= 0
\end{align*}
and
\begin{align*}
0 &\leq \liminf_{j\to\infty}(\cos d(x_{n_j},p)-\cos d(U_{n_j,l+2}x_{n_j},p)) \\
&= \liminf_{j\to\infty}(\cos d(x_{n_j},p)-\cos d(\alpha_{n_j,l+2}T_{l+2}U_{n_j,l+1}x_{n_j}\oplus(1-\alpha_{n_j,l+1})x_{n_j},p)) \\
&= \limsup_{j\to\infty}(\cos d(x_{n_j},p)-\cos d(\alpha_{n_j,l+2}T_{l+2}U_{n_j,l+1}x_{n_j}\oplus(1-\alpha_{n_j,l+1})x_{n_j},p)) \\
&\leq 0.
\end{align*}
Therefore
\[
\lim_{j\to\infty}(\cos d(x_{n_j},p)-\cos d(\alpha_{n_j,l+2}T_{l+2}U_{n_j,l+1}x_{n_j}\oplus(1-\alpha_{n_j,l+1})x_{n_j},p))=0.
\]
Using inequality $\sup_{j\in\mathbb{N}}d(x_{n_j}p)<\pi/2$, we have
\[
\lim_{j\to\infty}\dfrac{\cos d(x_{n_j},p)}{\cos d(\alpha_{n_j,l+2}T_{l+2}U_{n_j,l+1}x_{n_j}\oplus(1-\alpha_{n_j,l+1})x_{n_j},p)}=1.
\]
Since $\lim_{j\to\infty}d(x_{n_j},y_j^{(l+1)})=0$ and by Lemma \ref{Strongly quasinonexpansive sequence}, we have
\[
\lim_{j\to\infty}d(T_{l+2}y_j^{(l+1)},y_j^{(l+1)})=\lim_{j\to\infty}d(T_{l+2}y_j^{(l+1)},x_{n_j})=\lim_{j\to}d(T_{l+2}U_{n_j,l+1}x_{n_j},x_{n_j})=0.
\]
So, we have the hypothesis $k=l+1$, that is,
\[
\lim_{j\to\infty}d(x_{n_j},y_j^{(l+1)})=0,\quad \lim_{j\to\infty}d(T_{l+2}y_j^{(l+1)},y_j^{(l+1)})=0
\]
for $k=1,2,\ldots,r-1$. By induction, we obtain
\[
\lim_{j\to\infty}d(x_{n_j},y_j^{(k)})=0,\quad \lim_{j\to\infty}d(T_{k+1}y_j^{(k)},y_j^{(k)})=0
\]
for all $k=1,2,\ldots,r-1$. By Lemma \ref{subseq}, let $\{x_{n_{j_k}}\}$ be a $\Delta$-convergent subsequence of $\{x_{n_j}\}$ with the $\Delta$-limit $z$ such that $\lim_{k\to\infty}d(u,x_{n_{j_k}})=\liminf_{j\to\infty}d(u,x_{n_j})$. Then, since $T_1$ is $\Delta$-demiclosed and $\lim_{j\to\infty}d(x_{n_j},T_1x_{n_j})=0$, the $\Delta$-limit $z$ of $\{x_{n_{j_k}}\}$ belongs to $F(T_1)$. Similarly,  since $T_2$ is $\Delta$-demiclosed and $\lim_{j\to\infty}d(x_{n_j},y_j^{(1)})=\lim_{j\to\infty}d(y_j^{(1)}, T_2y_j^{(1)})=0$, $\{y_{j_k}^{(1)}\}$ is $\Delta$-convergent to $z$ and the $\Delta$-limit $z$ is belongs to $F(T_2)$. Using such techniques, we obtain $z\in F(T_i)$ for all $i=1,2,\ldots r$, and hence $z\in \bigcap_{i=1}^rF(T_i)=F$. Using Lemma \ref{conv} and the definition of the metric projection, we have
\begin{align*}
\liminf_{j\to\infty}d(u,W_{n_j}x_{n_j}) &= \liminf_{j\to\infty}d(u,\alpha_{n_j,r}T_rU_{n_j,r-1}x_{n_j}\oplus(1-\alpha_{n_j})x_{n_j}) \\
&= \liminf_{j\to\infty}d(u,\alpha_{n_j,r}T_ry_j^{(r-1)}\oplus(1-\alpha_{n_j})x_{n_j}) \\
&= \liminf_{j\to\infty}d(u,\alpha_{n_j,r}y_j^{(r-1)}\oplus(1-\alpha_{n_j})x_{n_j}) \\
&= \liminf_{j\to\infty}d(u,\alpha_{n_j,r}x_{n_j}\oplus(1-\alpha_{n_j})x_{n_j}) \\
&= \liminf_{j\to\infty}d(u,x_{n_j}) \\
&= \lim_{k\to\infty}d(u,x_{n_{j_k}}) \\
&\geq d(u,z) \\
&\geq d(u,P_Fu).
\end{align*}

Therefore, we obtain

\begin{align*}
\limsup_{j\to\infty}t_{n_j} &= \limsup_{j\to\infty}\left(1-\dfrac{\cos d(u,p)}{\sin d(u,W_{n_j}x_{n_j})\tan(2^{-1}\beta_{n_j}d(u,W_{n_j}x_{n_j}))+\cos d(u,W_{n_j}x_{n_j})}\right) \\
&= \limsup_{j\to\infty}\left(1-\dfrac{\cos d(u,p)}{0+\cos d(u,W_{n_j}x_{n_j})}\right) \\
&= 1-\dfrac{\cos d(u,p)}{\cos (\liminf_{j\to\infty}d(u,W_{n_j}x_{n_j}))} \\
&\leq 1-\dfrac{\cos d(u,p)}{\cos d(u,z)} \\
&\leq 0.
\end{align*}

By Lemma \ref{bankoku}, we have that $\lim_{n\to\infty}s_n=0$, that is, $\{x_n\}$ converges to $p=P_Fu$, and we finish the proof.
\end{proof}

\begin{remark}\normalfont
By Lemma \ref{delta-demiclosed}, a nonexpansive mapping defined on a CAT(1) space having a fixed point is quasinonexpansive and $\Delta$-demiclosed.
\end{remark}

\begin{remark}
In general, if $T_1,T_2,\ldots,T_r$ are nonexpansive, then $W$-mapping generated by $T_1,T_2,\ldots,T_r$ and $\alpha_1,\alpha_2,\ldots,\alpha_r$ is not necessarily nonexpansive.
\end{remark}

\section{Applications}\label{sec:Applications}

Let us recall some basic notation about functions on metric space. Let $X$ be a geodesic metric space and let $f$ be a function from $X$ into $(-\infty,\infty]$. We say $f$ is lower semicontinuous if the set $\{ x\in X\ |\ f(x)\leq a\}$ is closed for all $a\in\mathbb{R}$. The function $f$ is said to be proper if the set $\{x\in X\ |\ f(x)\neq\infty\}$ is nonempty. We say $f$ is convex if
\[
f(tx\oplus(1-t)y)\leq tf(x)+(1-t)f(y)
\]
for all $x,y\in X$ and $t\in(0,1)$. Let $X$ be a complete $\CAT(1)$ space such that $d(v,v')<\pi/2$ for every $v,v'\in X$. Let $f$ be a proper lower semicontinuous convex function from $X$ into $(-\infty,\infty]$. A resolvent of $f$ is defined by
\begin{align}\label{eq:resolvent-tansin}
R_fx:=\Argmin_{y\in X}\{f(y)+\tan d(y,x)\sin d(y,x)\}
\end{align}
in \cite{Kimura-Kohsaka}. Another type of the resolvent of $f$ is defined by
\begin{align}\label{eq:resolvent-logcos}
R_fx:=\Argmin_{y\in X}\left\{f(y)-\log\cos d(y,x)\right\}
\end{align}
in \cite{Kajimura-Kimura}. Both resolvents are quasinonexpansive, $\Delta$-demiclosed, and satisfy $F(R_f)=\Argmin_Xf$ (\cite{Kimura-Kohsaka,Kajimura-Kimura}). So, we can approximate a common minimizer of a finite number of functions by the following theorem.

\begin{theorem}\label{Convex function}
Let $X$ be a complete $\CAT(1)$ space such that $d(v,v')<\pi/2$ for every $v,v'\in X$. Let $f_1,f_2,\ldots,f_r$ be a finite number of convex function from $X$ into $(-\infty,\infty]$ such that $F:=\bigcap_{i=1}^r\Argmin_Xf_i\neq\emptyset$, and let $\alpha_{n,1},a_{n,2},\ldots,\alpha_{n,r}$ be real numbers for $n\in\mathbb{N}$ such that $\alpha_{n,i}\in[a,1-a]$ for every $i=1,2,\ldots,r$, where $0<a<1/2$. Let $R_{f_i}$ be a resolvent defined by either \eqref{eq:resolvent-tansin} or \eqref{eq:resolvent-logcos} for $i=1,2,\ldots,r$. Let $W_n$ be the W-mappings of $X$ into itself generated by $R_{f_1},R_{f_2},\ldots,R_{f_r}$ and $\alpha_{n,1},\alpha_{n,2},\ldots,\alpha_{n,r}$ for $n\in\mathbb{N}$. Let $\{\beta_n\}$ be a sequence of real numbers such that $0<\beta_n<1$ for every $n\in\mathbb{N}, \lim_{n\to\infty}\beta_n=0$, and $\sum_{n=1}^\infty\beta_n=\infty$. For given points $u,x_1\in X$, let $\{x_n\}$ be a sequence in $X$ generated by
\[
x_{n+1}=\beta_nu\oplus(1-\beta_n)W_nx_n
\]
for $n\in\mathbb{N}$. Suppose that one of the following conditions holds$:$
\begin{enumerate}
\item[(a)] $\sup_{v,v'\in X}d(v,v')<\pi/2;$
\item[(b)] $d(u,P_Fu)<\pi/4$ and $d(u,P_Fu)+d(x_1,P_Fu)<\pi/2;$
\item[(c)] $\sum_{n=1}^\infty\beta_n^2=\infty.$
\end{enumerate}
Then $\{x_n\}$ converges to $P_Fu$.
\end{theorem}

Let us consider a more specialized situation. For a closed convex subset $C$ of a complete $\CAT(1)$ space $X$, put
\[
i_C(x):=\left\{ \begin{array}{ll}
0 & (x\in C) \\
\infty & (x\notin C).
\end{array} \right.
\]
This function $i_C$ is a proper lower semicontinuous convex function. Thus the resolvent $R_{i_C}$ of $i_C$ is defined by either \eqref{eq:resolvent-tansin} or \eqref{eq:resolvent-logcos}, and it is quasinonexpansive and $\Delta$-demiclosed. In fact, we know $R_{i_C}=P_C$ and $F(R_{i_C})=\Argmin i_C=C$ for both definitions \eqref{eq:resolvent-tansin} and \eqref{eq:resolvent-logcos}. Thus we can apply Theorem \ref{Main} andhave an approximation of the nearest point in the intersection of finite family of closed convex subsets from a given point by using corresponding metric projection of each subset by the following theorem.

\begin{theorem}\label{Metric projection}
Let $X$ be a complete $\CAT(1)$ space such that $d(v,v')<\pi/2$ for every $v,v'\in X$. Let $C_1,C_2,\ldots,C_r$ be a finite number of closed convex subset of X such that $C:=\bigcap_{i=1}^rC_i\neq\emptyset$, and let $\alpha_{n,1},a_{n,2},\ldots,\alpha_{n,r}$ be real numbers for $n\in\mathbb{N}$ such that $\alpha_{n,i}\in[a,1-a]$ for every $i=1,2,\ldots,r$, where $0<a<1/2$. Let $W_n$ be the W-mappings of $X$ into itself generated by $P_{C_1},P_{C_2},\ldots,P_{C_r}$ and $\alpha_{n,1},\alpha_{n,2},\ldots,\alpha_{n,r}$ for $n\in\mathbb{N}$. Let $\{\beta_n\}$ be a sequence of real numbers such that $0<\beta_n<1$ for every $n\in\mathbb{N}, \lim_{n\to\infty}\beta_n=0$ and $\sum_{n=1}^\infty\beta_n=\infty$. For a given points $u,x_1\in X$, let $\{x_n\}$ be a sequence in $X$ generated by
\[
x_{n+1}=\beta_nu\oplus(1-\beta_n)W_nx_n
\]
for $n\in\mathbb{N}$. Suppose that one of the following conditions holds$:$
\begin{enumerate}
\item[(a)] $\sup_{v,v'\in X}d(v,v')<\pi/2;$
\item[(b)] $d(u,P_Cu)<\pi/4$ and $d(u,P_Cu)+d(x_1,P_Cu)<\pi/2;$
\item[(c)] $\sum_{n=1}^\infty\beta_n^2=\infty.$
\end{enumerate}
Then $\{x_n\}$ converges to $P_Cu$.
\end{theorem}

In the introduction we mention that there exists an example which is quasinonexpansive but not strongly quasinonexpansive. The following is such an example.

\begin{example}\label{example}\normalfont
A closed interval $[-1,1]$ is a complete $\CAT(1)$ space. Let $T:[-1,1]\to[-1,1]$ be defined by $Tx:=-x$. Then $F(T)=\{0\}$. It is easy to obtain that $T$ is quasinonexpansive and $\Delta$-demiclosed but it is not strongly quasinonexpansive.
\end{example}

\section*{Acknowledgment}
The first auther thanks Shin Nayatani and Shintarou Yanagida for valuable comments.

\end{document}